\documentclass[12pt]{amsart} \usepackage{amscd}
\usepackage[all]{xy}
\usepackage{amssymb}
\UseComputerModernTips

\newtheorem{theorem}{Theorem}
\newtheorem{lemma}[theorem]{Lemma}
\newtheorem{corollary}{Corollary}

\theoremstyle{definition}                 
\newtheorem{example}{Example}            \newtheorem{defn}{Definition}
 \newtheorem*{notation}{Notation}

\usepackage{amsfonts}
\newcommand{\field}[1]{\mathbb{#1}}          \newcommand{\Q}{\field{Q}}
                   \newcommand{\Z}{\field{Z}}
\newcommand{\C}{\field{C}}

\renewcommand{\a}{\alpha}
\renewcommand{\b}{\beta}                      \renewcommand{\d}{\delta}
\newcommand{\D}{\Delta}                    \newcommand{\e}{\varepsilon}
\newcommand{\g}{\gamma}                         \newcommand{\G}{\Gamma}
\renewcommand{\l}{\lambda}

\renewcommand{\th}{\theta}

 \newcommand{\fh}{\mathfrak h}

 \newcommand{\ra}{\rightarrow}

\begin{document}

\title[Differential Equations and Monodromy]  
{Differential Equations and Monodromy}

\author{T.N.Venkataramana}

{\address{ T.N.Venkataramana, School of Mathematics, TIFR, Homi Bhabha
Road, Colaba, Mumbai 400005, India}

\email{venky@math.tifr.res.in}

\date{}

\begin{abstract}  In these  expository notes,  we describe  results of
Cauchy, Fuchs and Pochhammer on  differential equations. We then apply
these  results   to  hypergeometric  differential  equation   of  type
$_nF_{n-1}$ and  describe Levelt's  theorem determining  the monodromy
representation explicitly in terms  of the hypergeometric equation. We
also give a brief overview, without  proofs, of results of Beukers and
Heckman,  on  the  Zariski  closure  of  the  monodromy  group  of  the
hypergeometric equation.  In the last  section, we recall  some recent
results  on thin-ness  and arithmeticity  of hypergeometric  monodromy
groups

\end{abstract}

\maketitle{}

\section{Introduction}\label{introsection}

In these notes,  we recall (in sections 1 to 3) the basic theory  of 
differential equations on the  unit disc and  on the punctured  unit 
disc. For  references see \cite{Cod-Lev} and \cite{Le}. \\

In  sections 4  and 5  we  apply the  theory developed  in sections  1
through 3  to (state and)  prove a result  of Levelt \cite{Le}  on the
monodromy   of   hypergeometric   differential   equations   of   type
$_nF_{n-1}$. \\

In the next few sections we use this description to prove some results
(some are  not proved  completely because the  proofs are  lengthy) of
Beukers  and Heckman  \cite{Beu-Hec}  on the  Zariski  closure of  the
monodromy  of the  foregoing hypergeometric  equation. In  particular,
they completely determine when the monodromy is finite. \\

{\bf Acknowledgments}  I thank  Professors Amarnath and  Padmavati for
inviting  me  to contribute  an  article  to  the proceedings  of  the
Telangana  Academy of  Sciences.  I  also thank  Professors F.Beukers,
Madhav Nori and P.Sarnak for very helpful conversations related to the
material presented in the paper. I am grateful to Max Planck Institute
for  Mathematics in  Bonn for  its hospitality  and financial  support
while this work was prepared for publication.

\newpage

\section{Monodromy Groups}

The  concept   of  monodromy   arises  in  many   seemingly  different
situations. We  will deal with some  of the simplest  ones, namely the
monodromy associated to linear  differential equations on open subsets
in  the complex  plane.

\subsection{Differential Equations on Open Sets in the Plane}

Let  $U$  be  a  connected  open  subset of  the  complex  plane.  Fix
holomorphic functions $f_i:U \ra \C$ with $0\leq i \leq n-1$. Consider
the differential equation
\[     \frac{d^n    y}{dz^n}     +     \sum    _{i=0}^{n-1}     f_i(z)
\frac{d^iy}{dz^i}=0.  \]  

If  $y_1,y_2$  are  solutions,  then  so  is
$c_1y_1+c_2y_2$ with $c_1,c_2\in \C$. That is, the space of solutions is
a vector space.  A fundamental result of Cauchy says that when $U$
is the unit  disc, there are holomorphic  functions $y_1, \cdots,
y_n$ on the  disc which are solutions of  this differential equation,
which are linearly independent and  such that all solutions are linear
combinations   of  these  solutions.    These  solutions   are  called
{\bf fundamental solutions}. 

\begin{theorem}  \label{cauchy}  (Cauchy)  Let  $f_0,\cdots,  f_{n-1}:
\Delta \ra \C$ be holomorphic functions on the open unit disc $\Delta$
and consider the differential equation
\[\frac{d^ny}{dz^n}+f_{n-1}(z)\frac{d^{n-1}y}{dz^{n-1}}+\cdots+
f_1(z)\frac{dy}{dz}+f_0(z)y=0.\]  Suppose  that $z_0,\cdots,  z_{n-1}$
are arbitrary complex numbers. Then there exists a solution $y$ to the
differential equation, which is holomorphic  in the whole of the disc,
such  that $\frac{d^jy}{dz^j}(0)=z_j$  for all  $j$ with  $0\leq j\leq
n-1$.\\

In particular, the differential equation has $n$ fundamental solutions.  
\end{theorem}

\begin{proof} We  first prove  this when $n=1$.  Suppose then  that we
have the equation
\[\frac{dy}{dz}+f_0(z)y=0,   \]  where   $f_0(z)=\sum  _{k=0}^{\infty}
a_kz^k$ a power  series which converges in $\mid  z \mid <1$.  Suppose
$y(z)=\sum _{k=0}^{\infty} x_kz^k$ is a formal power series with $x_k$
a  sequence of  elements of  $\C$. By  looking at  the  coefficient of
$z^{k-1}$  on  both sides  (which  are  formal  power series)  of  the
differential equation, it follows that, if the formal power series $y$
is to be a solution of  the differential equation, then the $x_k$ (for 
$k \geq 1$) must
satisfy the recursive relation 
\begin{equation}                \label{recursive}               -kx_k=
x_{k-1}a_0+x_{k-2}a_1+\cdots+x_0a_{k-1}.\end{equation} Let $R<1$; then
the convergence of $f_0(z)$ in $\mid  z \mid <1$ implies that there is
a constant  $M\geq 1$ such that  $\mid a_k\mid R^k <M$  for all $k\geq
0$.   Suppose $r<R$  is fixed.  Let, for  each $j$,  $M_j$  denote the
supremum
\[M_j=sup\{\mid  x_j\mid R^j, \mid  x_{j-1}\mid R^{j-1},  \cdots, \mid
x_1\mid R, \mid x_0\mid \} .\]
The equation (\ref{recursive}) shows that for each $k\geq 1$ we have 
\[k\mid         x_k\mid         \leq         \frac{M_{k-1}}{R^{k-1}}M+
\frac{M_{k-1}}{R^{k-2}}\frac{M}{R}+              \cdots              +
\frac{M_{k-1}}{R}\frac{M}{R^{k-2}}+M\frac{M}{R^{k-1}}=k\frac{M_{k-1}M}
{R^{k-1}}. \]  Therefore, $\mid  x_k \mid R^k  \leq M_{k-1}M$  for all
$k\geq  0$. In  particular, since  by  assumption $M\geq  1$, we  have
$M_k\leq  M_{k-1}M$  and   hence  $\frac{M_k}{M^k}$  is  a  decreasing
sequence, and hence a bounded sequence.  We may assume (increasing $M$
if necessary),  that $M_k\leq  M M^k$ for  all $k$.   Therefore, $\mid
x_k\mid r^k\leq \frac{MM^k}{R^k}r^k$. Therefore, if $\frac{Mr}{R} < 1$
then $\sum \mid x_k\mid r^k$  is dominated by the convergent geometric
series $M \sum (\frac{Mr}{R})^k$.  Hence the formal power series $\sum
x_kz^k$ converges in the smaller disc $\mid z \mid <\frac{R}{M}$. \\

We may similarly  solve the differential equation in  every small disc
inside the  unit disc $\Delta$ (  as a convergent  power series around
$z_0 \in \Delta$)  and by the uniqueness of the  power series - thanks
to the recursion (\ref{recursive}) -  the two power series coincide as
functions on the intersections of the smaller discs. Therefore, by the
principle of  analytic continuation,  there is a  holomorphic function
$y$  on all  of  the disc  which  is a  solution  of the  differential
equation $\frac{dy}{dz}+f_0(z)y =0$.  \\

Exactly the same proof shows that if now $y:\D \ra \C^n$ has values in
a  vector  space,  and $f_0$  is  replaced  by  a matrix  valued  {\it
holomorphic} function $A(z):\D \ra M_n(\C)$ (i.e.  an $M_n(\C)$-valued
convergent holomorphic function on $\D$)  then there is a power series
$y$  with coefficients  $x_k$ in  $\C^n$ which  is a  solution  of the
differential   equation   $\frac{dy}{dz}+A(z)(   y(z))=0$  and   which
converges on all of $\D$.  \\

Suppose  $y_1,y_2:   \D  \ra  M_n(\C   )$  are  two   solutions,  with
$y_1(0)=Id_n$  and $y_2(0)=x_0\in  GL_n(\C)$.  The  solution  $y_2$ is
uniquely determined by its  constant term $x_0\in M_n(\C)$, it follows
that $y_2$ is the product of  the matrix valued function $y_1$ and the
constant matrix $x_0$:
\[y_2=y_1x_0.\] This establishes that  every vector valued solution $y:
\D  \ra \C  ^n$  to  the equation  $\frac{dy}{dz}=A(z)y$  is a  linear
combination of the rows of the matrix $y_1$. \\

We now choose 
\[ A(z)=\begin{pmatrix} 0 & -1 & 0 & \cdots & 0 \\ 0 & 0 & -1 & \cdots
& 0 \\ \cdots &  \cdots & \cdots & \cdots & 0 \\ 0 &  0 & 0 & \cdots &
-1 \\ f_0(z) & f_1(z) & f_2(z) & \cdots & f_{n-1}(z)
\end{pmatrix} \quad  and \quad  y=\begin{pmatrix} y_1(z) \\  y_2(z) \\
\cdots \\  y_n(z) \end{pmatrix},  \] where $y$  is viewed as  a column
vector.  Then the  vector  valued equation  $\frac{dy}{dz}+A(z)y(z)=0$
yields  the  $n$ scalar  valued  equations  $y_ 1'(z)=y_2(z),  \cdots
y_{n-1}'(z)=y_n(z)$,                                                and
$y_n'(z)+f_0(z)y_1(z)+\cdots+f_{n-1}(z)y_n(z)=0$.    In  other  words,
$w=y_1(z)$ is the solution to the scalar valued differential equation
 \[\frac        {d^nw}{dz^n}+f_{n-1}(z)\frac{d^{n-1}w}{dz^{n-1}}+\cdots
f_1(z)\frac{dw}{dz}+f_0(z)w=0. \] This  proves Cauchy's theorem in all
cases.

\end{proof}

If $U$  is now taken to  be any connected  open set in $\C$,  then the
foregoing result  of Cauchy says  that at each  point $p$ of  the open
set, there are holomorphic functions  $y_1, \cdots, y_n$ defined in an
open neighborhood of $p$ which are fundamental solutions to the above
differential equation. If $\Gamma $  denotes a closed loop in the open
set $U$ starting  and ending at $p$, then  analytic continuation along
the path of  solutions is possible and when we  return to the original
point, we get new fundamental solutions $w_1, \cdots, w_n$. This means
that there is a matrix  $M=M(\gamma)$ depending on the path, such that
$w=My$ in a  neighborhood of $p$.  One can check  that the matrix $M$
depends only on  the homotopy class of the path  $\gamma$ based at $p$
and not on the path itself. \\

Moreover, if  $\gamma_1,\gamma _2$ are  {\it two} paths based  at $p$,
and $\gamma  $ is the composition  of these paths, then  one can check
that  $M(\gamma)=M(\gamma  _1)M(\gamma  _2)$.  Thus,  the  association
$\gamma  \ra   M(\gamma )$  yields  a  group   homomorphism  from  the
fundamental group of  the open set $U$ based  at $p$, into $GL_n(\C)$.
This homomorphism  is called the ``monodromy  representation'' and the
image is called the ``monodromy group''.

\subsection{Finiteness}

Let  $U \subset  \C$  be a  connected  open set.   Then  the space  of
holomorphic  functions   on  $U$  is   an  integral  domain   and  the
corresponding  field  of of  fractions,  i.e.   ratios of  holomorphic
functions, is  a field, called  the field of meromorphic  functions on
$U$.  Let  $U^* \ra U$  (given by $\tau  \mapsto z$) be  the universal
cover  $U^*$  of $U$  and  let  $\Gamma$  be the  deck  transformation
group. \\

We say that a function $f:U^*  \ra \C$ is {\bf algebraic} if it
satisfies   a   polynomial  relation 
$f^n   (\tau)   +\sum  _{i=0}^{n-1}   \phi _j (z)f^i(\tau)=0$  with  
coefficients  $\phi _j$ in  the  field  $K$ of  meromorphic
functions on $U$.

\begin{lemma} A function $f$ on  $U^*$ is algebraic if and only
if its orbit under the deck transformation group $\Gamma$ is finite.
\end{lemma}

\begin{proof} The polynomial relation holds  if $f$ is replaced by any
translate  under an  element $\g  \in \G$.  But since  there  are only
finitely many  roots to any polynomial,  it follows that  the orbit of
$f$ under $\Gamma$ is finite. \\

On the  other hand, if  a function $f$  on $U^*$ is  invariant under
$\Gamma$,  then  it  defines   a  holomorphic  function  on  the  base
$U$.  Therefore, if  the orbit  under $\Gamma  $ is  finite,  then the
polynomial  $P(t)=  \prod  _{\g\in  \Gamma /\Gamma  _f}  (t-\g(f))$  has
coefficients in $K$. Hence $f$ is algebraic.
\end{proof}

\begin{corollary} Suppose
\[\frac{d^ny}{dz^n}+f_{n-1}(z)\frac{d^{n-1}y}{dz^{n-1}}+
\cdots+f_1\frac{dy}{dz}+f_0y=0\]  is   a  differential  equation  with
coefficients  $f_i$ holomorphic  on  $U$. Suppose  that the  monodromy
representation is irreducible. Then a nonzero solution to the equation
is algebraic if and only if the monodromy is finite.
\end{corollary}

\begin{proof}  If $f$  is  a solution,  then  so is  $\g(f)$ for  $\g\in 
\G$. If the monodromy is finite,  it means that the orbit of $f$ under
$\Gamma $ is  finite, in particular, and hence it  is algebraic by the
lemma. \\

On the  other hand,  if some  solution $f$ is  algebraic, then  by the
lemma the  orbit is finite. It  means that the $\G$  translates of $f$
span a  subspace which  is $\Gamma $  stable and these  translates are
algebraic. By irreducibility, this is the whole space. This means that
for some basis of the space  of solutions, the orbit of $\G$ is finite
for every  element of the basis.  This means that the  image under the
monodromy representation of $\G$ is finite.
\end{proof}

\section{Punctured Disc}

Now consider the open set $U=\Delta ^*$, obtained by removing the point
$0$  from the  unit disc  $\Delta$. 

\begin{example}

Let us look at the differential equation
\[\frac{dy}{dz}-\frac{\alpha  }{z}y=0,  \] where  $\alpha  \in \C$  is
fixed.  Solving, we get  $y=z^{\alpha}$. This function is not ``single
valued''.  We view $z=e^{2\pi i \tau}=\phi (\tau)$ with $\tau \in \fh$
the upper half plane, with $\phi $ being a covering map.  Consider the
path $\omega  : [0,1] \ra  \fh$ starting at  $i$ and ending  at $i+1$.
Its composite $\gamma = \phi \circ \omega$ is a closed loop in $\Delta
^*$ based at  $p=e^{-2\pi}$; the effect of traversing  along this path
on the  solution $y$ is to  multiply it by $e^{2\pi  i \alpha}$.  Thus
$M(\gamma )$ is the $1\times 1$ matrix $e^{2\pi i \alpha}$.
\end{example}

\begin{example}
As another example, consider the equation 
\[\frac{d^2y}{dz^2}=- \frac{1}{z}\frac{dy}{dz}.\] Clearly the constant
function $y_1=1$  is a solution; it  is invariant under  the action of
the loop $\gamma$.

It  is easily  checked that  $y_2=\frac{1}{2\pi i}  log z$  is another
solution;   to  view   this   solution  as   a   function,  we   write
$y=\frac{1}{2\pi i} log (e^{2\pi i \tau})=\frac{1}{2\pi i} 2\pi i \tau
=\tau$;  hence the  action  of the  loop  $\gamma $  of the  preceding
example,  is to  take $y_2$  into  the new  solution 
$y_2+\frac{1}{2\pi i} 2\pi  i=y_2+y_1$. Hence
\[M(\gamma )=\begin{pmatrix} 1 & 1 \\ 0 & 1 \end{pmatrix}.\]  
\end{example}

\subsection{Regular Singular Points and a Theorem of Fuchs}

We now  look at a  more general case;  suppose we have  a differential
equation of the form
\begin{equation}                                \label{regularsingular}
\frac{d^ny}{dz^n}+f_{n-1}(z)\frac{d^{n-1}y}{dz^{n-1}}+           \cdots
+f_1(z)\frac{dy}{dz}+f_0(z)y,\end{equation} where each $f_i(z)$ has at
most a pole of order $n-i$ at $z=0$. Then the monodromy $M$ (i.e.  the
action of the  generator of $\pi _1(\D ^*)\simeq \Z$)  acts on the $\C
^n$  space of  solutions.   Write $\theta  =z\frac{d}{dz}$; using  the
relation  $\theta  ^2=z^2\frac{d^2}{dz^2}+\theta$  one  can  show,  by
induction, that
\[z^k\frac{d^k}{dz^k}= \th (\th-1)\cdots (\th-k+1) \] for every $k\geq
1$. Then  the differential  equation (after multiplying  throughout by
$z^n$), takes the form
\[\th  (\th -1)\cdots  (\th -n+1)y+  zf_{n-1}(\th (\th  -1)\cdots (\th
-n+1)y+\]
\[+ \cdots + z^{n-2}f_2 \th (\th -1)y + z^{n-1}f_1\th y+ z^nf_0y  =0.  \]
Rewriting this yields
\[\th  ^n y+ F_{n-1}(z)  \th ^{n-1}+F_{n-2}  (z) \th  ^{n-2}y+\cdots +
F_1(z)  \th   y+F_0(z)y=0  \]  where  now  the   functions  $F_i$  are
holomorphic  on  all of  the  disc  (including  the puncture).   Write
$a_i=F_i(0)$   and   $f(t)=t^n+a_{n-1}t^{n-1}+\cdots   +a_1t+a_0=\prod
_{j=1}^n  (t-\a  _j)$.   The  equation  $f(t)=0$ is  called  the  {\bf
indicial equation} and the roots  of the indicial equation (i.e. roots
of the polynomial $f$) are called the indicial roots.

\begin{theorem}  \label{Fuchs} (Fuchs)  With  the preceding  notation,
assume  that $0$  is  a  regular singular  point  of the  differential
equation (\ref{regularsingular}).  Then, the characteristic polynomial
of   the   monodromy  matrix   $M$  of   the  differential   equation
(\ref{regularsingular}) is the polynomial
\[\prod _{j=1}^n  (t-e^{2\pi i  \a_j}).\] Moreover, every  solution of
the   differential  equation   (\ref{regularsingular})  is   a  linear
combination of functions  of the form $\phi (z)z^{\a}  P(log z)$ where
$\phi  $ is  a holomorphic  function on  all of  the disc,  $\a$  is a
complex number and $P$ is a polynomial.
\end{theorem}

Theorem  \ref{Fuchs} will be  recast in  terms of  matrix valued
solutions and  the differential equation  (\ref{regularsingular}) will
be rewritten as a {\it first order} equation.

\subsection{First order Matrix Valued Differential Equations}

Suppose now that $A:\Delta \ra M_n(\C)$ is a holomorphic map on all of
the  disc. Then  $A(z)$ is  represented by  a convergent  power series
$A(z)=\sum A_kz^k$ where $A_k\in M_n(\C)$. We look for local solutions
$Y:   \Delta   ^*  \ra   M_n(\C)$   to   the   first  order   equation
$z\frac{dY}{dz}=A(z)Y(z)$. \\

\begin{notation} \label{matrixexponents}  If $T\in M_n(\C)$  and $z\in
\Delta  ^*$  we  write  $z^T$   for  the  matrix  represented  by  the
(convergent)  exponential power  series in  the matrix  variable $(log
z)T$:
\[z^T=exp ((log z)T)=\sum _{k=0}^{\infty} \frac{(log z)^kT^k}{k!}\]
\end{notation}

We list some properties of the matrix exponent. 

[1] If  $A$ and $B$  are commuting square  matrices of size  $n$, then
$z^{A+B}=z^Az^B$.

[2] If $A\in M_n(\C)$ and $g \in GL_n(\C)$, then $z^{gAg^{-1}}= gz^Ag^{-1}$.  

[3] If $N\in M_n(\C)$ is nilpotent, then $z^N$ is a polynomial in $log z$. 

[4] If $A \in M_n(\C)$ is a diagonal matrix whose diagonal entries are
$a_1,a_2,  \cdots,a_n$ then  $z^A$  is also  a  diagonal matrix  whose
diagonal entries are $z^{a_1}, z^{a_2}, \cdots, z^{a_n}$.

[5] These properties imply that  if $A\in M_n(\C)$ is any matrix, then
the entries of  the matrix $z^A$ are linear  combinations of functions
of the  form $z^{\a} P(log z)$ where  $P$ is a polynomial  and $\a \in
\C$ is a fixed complex number.

[6] The derivative of $z^A$ satisfies:  $z\frac{dz^A}{dz}= z^AA$. 

[7] The monodromy operator on the multivalued function $z^A$ is simply
$e^{2\pi i A}$, since $z^A=e^{2\pi i \tau A}$ and the generator of the
Deck transformation group takes $\tau $ to $\tau +1$.

For a reference to the following see \cite{Cod-Lev}, Theorem (4.1) and
Theorem (4.2). 

\begin{theorem}  \label{matrixfuchs} (Fuchs)  Suppose  $A: \Delta  \ra
M_n(\C)$ is  a holomorphic  function. Let $\fh  \ra \Delta ^*$  be the
exponential covering map as before. Consider the differential equation
in $Y(z)=Y^*(\tau )\in M_n(\C)$:
\begin{equation}                                     \label{firstorder}
\frac{dY}{dz}=\frac{A(z)}{z}Y.\end{equation}  The   monodromy  of  the
equation acts on the space of solutions $Y^*$ by the formula $Y^*(\tau
+1)=Y^* (\tau  )M$ where  $M\in GL_n(\C)$.  Moreover,  the semi-simple
part  of the matrix  $M$ is  conjugate to  the exponential  $e^{2\pi i
A_s}$  of $A_s$, where  $A_s$ is  the semi-simple  part of  the matrix
$A_0= A(0)$.
\end{theorem}

\begin{proof} First  assume that if $\lambda, \mu$  are {\it distinct}
eigenvalues  of  the matrix  $A_0$,  then they  do  not  differ by  an
integer. This  means that no eigenvalue of  the adjoint transformation
$ad  A_0$ can be  a nonzero  integer.  We  then show  that there  is a
holomorphic function $X:\Delta \ra GL_n(\C)$ such that
\[Y(z)=X(z)z^{A_0}\]  is  a  solution  of the  differential  equation
(\ref{firstorder}).  Write $X(z)=\sum _{k=0}^{\infty}  X_kz^k$, and
solve  for   the  coefficients   $X_k$.  Write,  as   before,  $\theta
=z\frac{d}{dz}$.  Then the  differential equation  for $Y$  is $\theta
Y(z)= A(z)Y(z)$; moreover, by  the formula for the differentiation for
a product, we get
\[\theta Y(z)=\theta (X(z)z^{A_0})= \theta (X(z)) z^{A_0}+ X(z)z^{A_0}
A_0=\] \[= A(z)Y(z)=A(z)X(z)z^{A_0}.\] We now cancel $z^{A_0}$ on both
sides of the preceding equation and obtain
\[\theta (X(z))+X(z)A_0= A(z)X(z), \]  where now $A$ is holomorphic on
all of  the disc and $X$  is assumed to  be holomorphic on all  of the
disc.  Writing the  power series  for $X$  and $A$,  we then  get, for
$k\geq 1$, the recursion
\[ kX_k+ X_kA_0=A_0X_k+\sum _{j=0}^{k-1} A_{k-j}X_j, \] and for $k=0$,
the equation $X_0A_0=A_0X_0$.  We can solve for $X_0$  by taking $X_0$
to be identity. The recursion for the coefficients is
\[(k-adA_0)X_k=\sum _{j=0}^{k-1}A_{k-j}X_j.\]  This can be  solved for
all $k\geq 1$ since, by  assumption, {\it non-zero integers $k$ cannot
be eigenvalues  of the operator  $adA_0$}; therefore, $k-adA_0$  is an
invertible operator and hence $X_k$ may be written as a combination of
the $X_j: j\leq k-1$. \\

[ We now check that the formal power series $\sum X_kz^k$ converges in
a  small   enough  neighborhood   of  $0$.   Consider   the  sequence
$1-\frac{adA_0}{k}$ for  $k\geq 1$. For  $k$ large enough,  the $k$-Th
term of this sequence is  close to the identity matrix; by assumption,
all the  terms of this  sequence are non-singular. Hence  the sequence
$(1-\frac{adA_0}{k})^{-1}$ is  bounded from above by  a constant $M>1$
say. Since  the sequence $A_{k-j}R^{k-j}$  ($k\geq j$) is  bounded, we
may assume that  $\mid A_{k-j}\mid R^{k-j}\leq M$ for  all $k,j$. Let,
as in  the proof  of Cauchy's  theorem, $M_k$ be  the supremum  of the
matrix norms $\mid X_j\mid R^j$  for $j\leq k$. The recursive relation
for the $X_k$ now implies that
\[k  \mid X_k\mid  \leq  M  Mk_{k-1} .  \]  Therefore, $\mid  X_k\mid
R^k\leq M^2M_{k-1}$. Since  $M\geq 1$, we also have  $\mid X_j\mid R^j
\leq M^2M_{k-1}$  for all $j\leq k-1$. Hence  $M_k\leq M^2M_{k-1}$ and
therefore, $M_km^{-2k}$  is a decreasing  sequence and is  bounded. We
may assume then that $\mid X_k\mid R^k\leq M_k\leq M M^{2k}$ and hence
$\sum X_kz^k$ converges if $\mid z \mid < \frac{R}{M^2}$.]  \\

Thus  the  monodromy  action  on $Y(z)=X(z)z^{A_0}$  is  simply  right
multiplication  by the exponential  matrix $e^{2\pi  i A_0}$  of $A_0$
since the solution  $X(z)$ is holomorphic also at  the puncture and is
invariant under the  monodromy action. This proves the  Theorem in the
case when  distinct eigenvalues of  $A_0$ remain distinct  modulo $1$.
\\

The proof  of the general case of  the Theorem can be  reduced to this
case.   Fix an  eigenvalue $\l$  of the  linear  transformation $A_0$.
Write $\C ^n= E\oplus F$  where $E$ is the generalized $\l$ airspace
for $A_0$, and $F$ an $A_0$ stable supplement to $E$. If $\e_1,\cdots,
\e_r$ is a  basis of $E$, and $\e_{r+1},\cdots, \e_n$  a basis of $F$,
then with  respect to  the basis $\e_1,  \cdots, \e_n$ of  $\C^n$, the
matrix of  the transformation which is  $z$ times the  identity on $E$
and  identity on $F$  is given  by $\begin{pmatrix}  zI_r &  0 \\  0 &
I_{n-r}\end{pmatrix}$, where $I_k$ is the identity matrix of size $k$.
Moreover,  $A_0=\begin{pmatrix}   \l  I_r   +N_r  &  0   \\  0   &  \d
_0\end{pmatrix}$ where $\d_0$ acts on $F$ and $N_r$ is a {\it nilpotent}
matrix of  size $r$.   Write $Y(z)=  \begin{pmatrix} zI_r &  0 \\  0 &
I_{n-r}\end{pmatrix}  W(z)=M(z)W(z)$.   Then it  is  easily seen  that
$W(z)$ satisfies the equation
\[ \theta  W(z)= B(z)W(z)\] where  $B(z)$ is holomorphic on  $\Delta $
and  $B_0=B(0)=  \begin{pmatrix}  (\l -1)I_r  +N_r  &  \b  _0  \\ 0  &  \d
_0\end{pmatrix}$.   Thus the  semi-simple part  of the  exponential of
$B_0$ is  conjugate to that of  $A_0$. Moreover, the  monodromy of $W$
and of $Y$  are the same. Consequently, $Y$ may be  replaced by $W$ in
the statement of the theorem without altering the conclusion. \\

We now apply  the preceding repeatedly to ensure that  if $\l$ and $\l
'$ are two  distinct eigenvalues of $A_0$ which differ  by an integer,
the  $A_0$ is  replaced by  $B_0$ such  that these  eigenvalues become
equal.   That is,  suppose $\lambda  =\lambda '+m$  for some  positive
integer  $m$ say.  As above,  we replace  $\lambda $  by $\lambda  -1$
without altering the monodromy; we do  this $m$ times until $\lambda $
is replaced by $\lambda '$, with monodromy unchanged. \\

Applying this procedure repeatedly to all eigenvalues which differ by
an integer,  we can thus ensure  that all the distinct  eigenvalues of
$A_0$ remain  distinct modulo  $1$. Then  we are  in the  special case
where $adA_0$ does not have  non-zero integers as eigenvalues. In that
case, the Theorem has already been proved.

\end{proof}

\subsection{Complex Reflections}

For a reference to the following, see Theorem 3.1.2 of \cite{Beu}. 

\begin{theorem} \label{complexreflection} {\rm (}Pochhammer{\rm )}
If as before, we have a differential equation 
\[     \frac{d^n    y}{dz^n}     +     \sum    _{i=0}^{n-1}     f_i(z)
\frac{d^iy}{dz^i}=0,  \] on  the punctured  disc $\Delta  ^*$,  and we
assume that the  functions $f_i$ have at most a  simple pole at $z=0$,
then  there are $n-1$  solutions which  extend holomorphic ally  to the
puncture, and  one solution which (possibly) has  singularities at the
puncture. Moreover, the monodromy matrix is of the form
\[M=\begin{pmatrix} 1 & 0 &  0 & \cdots & * \\ 0 & 1  & 0 & \cdots & *
\\ \cdots & \cdots & \cdots \\0 &  0 & 0 & \cdots & c \end{pmatrix} \]
for  some  $c\neq  0$.  
\end{theorem}

The number $c$ is called the {\it exceptional eigenvalue} (it can even
be 1) and  the matrix $M$ is called a {\it  complex reflection} (it is
identity on a co dimension one subspace of $\C ^n$). \\

\begin{proof}  By   the  same
procedure as  before, the  differential equation of  order $n$  can be
converted to a differential equation of order $1$ but with
solutions in the vector space $\C ^n$:
\[\frac{dy}{dz}= A(z)y(z)=\begin{pmatrix} 0 & 1 &  0 & \cdots & 0 \\ 0
& 0 & 1 & \cdots & 0 \\ \cdots & \cdots & \cdots & \cdots \\ -f_0(z) &
-f_1(z) & -f_2(z) & \cdots & -f_{n-1}(z) \end{pmatrix} (y(z)) .\]

We write  the vector valued  formal power series  expansion $y(z)=\sum
x_kz^k$,     with     $x_k$     in     $\C^n$.     If     we     write
$A(z)=\frac{A_{-1}}{z}+A_0+A_1z+\cdots   +   A_kz^k+\cdots   $,   then
$A(z)-\frac{A_{-1}}{z}$ is  a convergent power series  (with values in
$M_n(\C)$) in $\mid z \mid <1$.  Solving term by term, for each $k\geq
0$ comparing the coefficient of $z^k$ we get (cf \ref{recursive})
\begin{equation}                                          \label{recur}
kx_k=A_{-1}x_k+A_0x_{k-1}+A_1x_{k-2}+\cdots  +A_kx_0 .  \end{equation}
Now the first $n-1$ rows of the matrix $A_{-1}$ are all zero. the only
other eigenvalue $d$ of $A_{-1}$  is the residue of $-f_{n-1}$ at $0$.
If $d$ is never a  positive integer, then $A_{-1}-k$ is invertible for
{\it all}  positive integers $k$. Hence the above  recursion shows
that  all the  $x_k;k\geq 1$  are  uniquely determined  by $x_0$;  the
equation   for  $k=0$   shows  that   $x_0$  satisfies   the  equation
$A_{-1}x_0=0$.   This is  a co-dimension one  subspace of  $\C^n$ and
hence the space of  holomorphic solutions of the differential equation
is of dimension at least $n-1$. This proves the first part. \\

If a solution is holomorphic,  then analytic continuation along a loop
around  $0$ does  not  change  the function  and  hence the  monodromy
element acts trivially. This proves the second part of the Theorem.\\

Slightly more  work is  needed when the  eigenvalue $d$ is  a positive
integer. The equation  (\ref{recur}) may be applied to  all $k\neq d$;
in particular, if $A_{-1}x_0$ is zero, then $x_1, \cdots, x_{d-1}$ are
uniquely  determined. However, the  equation (\ref{recur})  applied to
$k=d$ shows that there exists  a linear transformation $B_d$ on $\C^n$
such that for each $x_0\in Keri (A_{-1})$ we have
\[(d-A_{-1})(x_d)=  B_d(x_0)  .\]  The  image $W$  of  $d-A_{-1}$  has
dimension $n-1$  and hence if $B_d(x_0)$  lie in this  image $W$, then
the  recursion (\ref{recur})  still applies  to locate  an  $x_d$. The
other $x_j  (j\geq k+1)$ are now uniquely  determined by (\ref{recur})
and  hence the space  of solutions  which are  holomorphic at  $0$ has
dimension at least $n-2$: this is  the dimension of the space of $x_0$
satisfying $A_{-1}x_0=0$ and $B_k(x_0)\in W=Image(d-A_{-1})$. \\

Now   we  take   $x_0=0$.   Then  by   (\ref{recur}),  $x_1=\cdots   =
x_{d-1}=0$.  Moreover,  $x_d$  satisfies $(A_1-d)(x_d)=0$.  Since  the
kernel of $A_1-d$ has dimension one, there does exist a non-zero $x_d$
with this property. Then by (\ref{recur}), all $x_j$  with $j\geq d+1$
are determined  and hence there  exists an extra  holomorphic solution
$w$  of the  form  $w(z)=z^DC_d+z^{d+1}x_{d+1}+\cdots+x_kz^k+\cdots $
. Hence  the space of holomorphic  solutions is again  of dimension at
least  $n-1$.  This proves  the  first part  when  $d$  is a  positive
integer.  The statement about the monodromy matrix follows as before.
\end{proof}

\subsection{The Plane with Two Punctures}

Consider  the   twice  punctured  plane   ${\mathcal  U}=\C  \setminus
\{0,1\}$, and a differential equation

\[\frac{d^n      y}{dz^n}     +      \sum      _{i=0}^{n-1}     f_i(z)
\frac{d^iy}{dz^i}=0, \]  where $f_i:\C  \setminus \{0,1\} \ra  \C$ are
holomorphic. Now  the fundamental  group of $\mathcal  U$ is  the free
group $F_2$  on two generators  $h_0,h_1$, given by small  loops going
counterclockwise  once  around $0$  and  $1$  respectively.  Thus  the
monodromy   representation   of  the   differential   equation  is   a
homomorphism from $F_2$ into $GL_n(\C)$; it is completely described by
specifying what the images of  $h_0$ and $h_1$ are. Thus the monodromy
is described by giving two matrices in $GL_n(\C)$. \\

The open set ${\mathcal U}=\C \setminus \{0,1\}$ may also be viewed as
${\mathbb P}^1\setminus  \{0,1,\infty\}$. Thus, the  fundamental group
of $\mathcal U$ can also be thought  of as the free group on the small
loops  $h_{\infty}, h_0,h_1$  going  around $\infty,  0,1$ modulo  the
relation $h_{\infty}h_1h_0=1$. Denote, by $A$ the image of
$h_{\infty}$ and by $B^{-1}$ that of $h_0$. Then $C=A^{-1}B$. \\

Since the universal cover of  $\C \setminus \{0,1\}$ is the upper half
plane, it  follows that the  solutions to the foregoing  equations are
functions on  the upper half plane  and that the  fundamental group of
$\mathcal U$ is the deck transformation group.

\newpage 

\section{The Hypergeometric Differential Equation} 
\label{hypergeometric}

Suppose  that  ${\mathcal  U}={\mathbb P}^1\setminus  \{0,1,\infty\}$.
Put $\theta  =z\frac{d}{dz}$ and let $\a  _1, \cdots, \a  _n$, and $\b
_1, \cdots, \b _n$ be complex numbers. Write
\begin{equation} \label{hypergeom} D=(\theta +\b _1-1)\cdots (\theta
+\b _n-1) -z(\theta +\a _1)\cdots (\theta +\a _n). \end{equation} This
is a  differential operator  on $\mathcal U$.  The equation  $Dy=0$ is
called the ``hypergeometric  differential equation'' and the solutions
are called  ``hypergeometric functions''.  These are  functions on the
upper half plane.

\begin{theorem}  \label{hypergeometricmonodromy}  Under the  monodromy
representation considered  in the preceding  subsection, the monodromy
of  the generator  $h_0$ around  the puncture  $0$  has characteristic
polynomial $\prod (t-e^{2\pi i(1-\b _j)})$ and the monodromy action of
$h_{\infty}$  has  characteristic polynomial  $\prod  (t-e^{2\pi i  \a
_j})$. Moreover, the element $h_1$ acts by a complex reflection.
\end{theorem}

\begin{proof}  Since  the   hypergeometric  differential  equation  is
already written  in the  ``$\theta $'' form,  and the  coefficients of
powers of $\theta $ are linear polynomials in $z$, it follows that $0$
is a regular singular point  of the differential equation $Du=0$ where
$D$ is the operator in (\ref{hypergeom}). The indicial equation at $0$
is  thus $\prod  _{i=1}^n(t +  \b_j-1))=0$.  By the  Theorem of  Fuchs
(Theorem  \ref{Fuchs}), it  follows that  the monodromy  of  $h_0$ has
characteristic polynomial $\prod (t-e^{2\pi i (1-\b _j)})$. \\

We now  consider the point $\infty$;  by changing the  variable $z$ to
the  variable $w=\frac{1}{z}$,  the operator  $\th  _z= z\frac{d}{dz}$
changes to $-\th _w=  -w\frac{d}{dw}$.  Multiplying throughout by $w$,
the operator $D$ changes to
\[w(-\th _w+\b_1-1)\cdots (-\th _w+\b _n-1)- (-\th _w+\a_1)\cdots (-\th
_w +\a_n),\] which is  just a constant multiple of the hypergeometric
operator
\[D'=(\th _w-\a_1)\cdots  (\th _w-\a _n)-w(\th  _w+1-\b _1)\cdots (\th
_w+1-\b _n). \]  Therefore, $\infty$ is also a  regular singular point
of the equation  $Du=0$ and the monodromy statement  follows as in the
preceding paragraph.  \\

Consider  now  the point  $z=1$.  We write  out  the  operator $D$  of
(\ref{hypergeom}) (which is in ``$\theta $`` form) in terms of powers
of $\frac{d}{dz}$: this is of the form
\[D=                                               z^n\frac{d^n}{dz^n}+
P_{n-1}(z)\frac{d^{n-1}}{dz^{n-1}}+\cdots+P_0(z) \]
\[
-z(z^n\frac{d^n}{dz^n}+Q_{n-1}(z)\frac{d^{n-1}}{dz^{n-1}}+\cdots+Q_0(z))\]
where $P_i,Q_i$ are polynomials in $z$.  Therefore,
\[D=z^n(1-z)\frac{d^n}{dz^n}+R_{n-1}(z)
\frac{d^{n-1}}{dz^{n-1}}+\cdots+R_0(z),\]
where the $R_i(z)$ are polynomials.  Hence the hypergeometric equation
$Dy=0$ at $z=1$ (after normalising the highest coefficient to be $1$),
has the  property that all  its coefficients $\frac{R_i(z)}{z^n(1-z)}$
at  $z=1$  have   at  most  a  simple  pole   at  $z=1$.   By  Theorem
\ref{complexreflection}  it  follows  that  $h_1$ maps  to  a  complex
reflection.
\end{proof}

The following  theorem says that  these facts suffice  to characterise
the monodromy action.

\subsection{Statement of Levelt's Theorem}

The monodromy  representation is  very simply described.  Suppose that
$\a _j-  \b _k$ is  not an integer  for any two suffices  $j,k$. Write
$f(x)=\prod  _{j=1}^n (x-e^{2\pi i  \a _j}),  g(x)=\prod _{k=1}^n
(x-e^{2\pi  i  \b  _k})$.   These  are  monic  polynomials  of  degree
$n$.   Write  $f=x^n+a_{n-1}x^{n-1}+\cdots+a_0$.  The   quotient  ring
$R=\C[x]/(f(x))$ is a  vector space of dimension $n$  and has as basis
the vectors $1,x,\cdots,x^{n-1}$. Write $A$ for the linear operator on
the  ring $R$  given by  multiplication by  $x$. With  respect  to the
foregoing basis, the matrix of $A$ is
\[A= \begin{pmatrix} 0 & 0 & 0 &\cdots &  -a_0 \\ 1 & 0 & 0 & \cdots &
-a_1 \\ 0 & 1 & 0 & \cdots & -a_2 \\ \cdots & \cdots & \cdots & \cdots
\\ 0 &  0 & \cdots & 1  & -a_{n-1} \end{pmatrix} \] and  is called the
``companion  matrix'' of  $f$.  Similarly, let  $B$  be the  companion
matrix of  $g$. Note that $C=A^{-1}B$  is identity on  the first $n-1$
basis vectors. Hence $C$ is a complex reflection.

\begin{theorem}  \label{levelt}  (Levelt,1960)  There exists  a  basis
$\e_1,\cdots, \e_n$  of the space  of solutions to  the hypergeometric
equation  such  that  the  monodromy  representation  sends  $h_0$  to
$B^{-1}$ and $h_{\infty}$ to $A$. \\

Moreover,  if  $\rho  $  is  any  representation  of  the  free  group
$h_0,h_{\infty}$   into   $GL_n(\C)$   such   that   the   images   of
$h_{\infty},h_0^{-1}$ have characteristic  polynomials $f,g$, and such
that  $h_1$ goes to  a complex  reflection, then  $\rho$ is  the above
monodromy representation.
\end{theorem}

The  representation  described  in  Levelt's  theorem  is  called  the
``hypergeometric representation'' and the monodromy group is called a
``hypergeometric''.

We prove Levelt's theorem in the next section.

\section{Levelt's Theorem} \label{levelt}

\subsection{Notation}

Denote  by $R_0$  the ring  $\Z[x_i^{\pm 1},y_i^{\pm  1}]$  of Laurent
polynomials in the variables $x_1, \cdots, x_n$ and $y_1, \cdots, y_n$
with integral coefficients,  and by $K_0$ its quotient  field. Let $R$
denote  the sub-ring of  $R_0$ generated  by the  elementary symmetric
functions $\sigma _i$ in  $x_i$ and the elementary symmetric functions
$\tau _j$ in $y_j$ together with the inverses $\sigma _n ^{-1},\tau _n
^{-1}$.  Denote by  $K$  the  quotient field  of  $R$; then  $K\subset
K_0$. Put
\[f= f(t)=\prod _{i=1}^n (t-x_i)  =t^n+ \sum _{i=1} ^{n-1}A_i t^{n-i},\]
\[g= g(t)= \prod _{j=1}^n (t-y_j)=t^n+\sum _{i=1}^n B_it^{n-i}.\] Then
$f,g$ are polynomials  in $t$ with coefficients in  $R$. $F_2$ denotes
the  free group  on two  generators (which  in the  sequel,  are often
written   $h_0,h_{\infty}$).   Define   a  representation   $\rho$  on
$F_2=<h_0,h_{\infty}>$ by $h_0\mapsto  A$ and $h_{\infty}^{-1} \mapsto
B$ where  $A,B$ are companion matrices of  $f,g$ respectively.  Denote
by $\Gamma = \Gamma (f,g)$  the group generated by $A,B$ in $GL_n(R)$,
and by $G$ the Zariski closure  of $\Gamma$ in $GL_n$. As usual, $G^0$
denotes the  connected component  of identity of  $G$; it is  a normal
subgroup  in  $G$   of  finite  index.  Denote  by   $\Gamma  ^0$  the
intersection of $\Gamma $ with $G^0$.  \\

The  group   $\Gamma$  is   called  the  {\it   hypergeometric  group}
corresponding  to the  parameters $x_i,y_j$.  Sometimes, it  is simply
called the {\bf hypergeometric} corresponding to the polynomials $f,g$
above. \\

Let $\pi  : R \ra S$ a  ring homomorphism with $S$  an integral domain
whose quotient field  is denoted $K_S$. Denote by  $a_i, b_i\in S$ the
images  of  $x_i,y_i$ under  the  map  $\pi :  R  \ra  S$.  Denote  by
$f_S,g_S$ the monic polynomial in $t$ given by
\[f_S(t)=\prod   _{i=1}^n   (t-a_i),   \quad   g_S(t)=\prod   _{i=1}^n
(t-b_i).\]  We view the  free $S$  module $S^n$  as the  quotient ring
$S[t]/(f_S(t))$. With  respect to  the basis $1,t,\cdots,  t^{n-1}$ of
$S^n$,  $A_S$ is  simply the  matrix  of the  ``multiplication by  $t$''
operator.  Denote  by $k$ the  g.c.d. of $f_S$  and $g_S$. $B_S$  is the
operator which acts as follows: $B_S(t^i)=A_S(t^i)=t^{i+1}$ if $i\leq n-2$
and   $B_S(t^{n-1})=   t^n-g_S(t)$.    Let   $W$   be   the   ideal   of
$S^n=S[t]/(f_S(t))$ generated  by the polynomial $k$; then  $W$ is $A_S$
stable.  Hence so is $W\otimes K_S$.

\begin{lemma} \label{irreducible}  The subspace $W_S=W\otimes  K_S$ is
also  $B_S$ stable  and under  the action  of $A_S,B_S$,  the subspace
$W_S$ is irreducible. \\

Moreover, on  the quotient $V_S/W_S$ the  operators $A_S,B_S$ coincide
and  as  a module  over  $A$  (multiplication  by $t$),  the  quotient
$V_S/W_S$ is the ring $K_S[t]/(k(t))$.\\

In  particular, if  $f,g$ are  co-prime  (i.e. $a_i\neq  b_j$ for  any
$i,j$), then $V_S=W_S$ is irreducible for the action of $F_2$.
\end{lemma}

\begin{proof} We temporarily  write $A_S=A,B_S=B$.  Put $D=A-B$. Then,
the  image  of  $D$ on  $V_S$  is  the  line generated  by  $f_S-g_S$.
Moreover, $D$  is zero on  the monomials $1,t,\cdots, t^{n-2}$  and is
$f-g$ on $t^{n-1}$. Therefore, the  image under $D$ of a polynomial of
degree exactly $n-1$ is a non-zero multiple of $f-g$. \\

Any  subspace  of  $V_S$  which   is  stable  under  $A$  contains  an
eigenvector  for  $A$;  these  eigenvectors  are of  the  form  $\e_i=
\frac{f(t)}{x-a_i}$ for some $i$.  This a polynomial of degree exactly
$n-1$.  Hence $D(\e_i)$  is a  non-zero multiple  of $f-g$;  thus any
subspace of $V_S$ which is stable under $A,D$ contains $f-g$ and hence
contains $W_S=K_S[t](f-g)+K_S[t]f=K_S[t]k(t)=(k(t))$.  This proves the
first part of the Lemma. \\

On the quotient  $V_S/W_S$, the operator $D$ is  zero, since the image
of   $D$    lies   in   $W_S$.    Hence   $A=B$   on    the   quotient
$V_S/W_S=K_S[t]/(k(t))$. This proves the second part. \\

The third part is a corollary of the first part. 
\end{proof}

\begin{theorem} \label{Levelt} (Levelt) Suppose that $a_i\neq b_j$ for
any $i,j$.   Suppose $h_0\mapsto a$ and  $h_{\infty}^{-1}\mapsto b$ is
any  other   irreducible  representation   $\rho  '$  of   $F_2$  into
$GL_n({\overline K_S})$  such that the following  {\rm two} conditions
hold.   (1)  the characteristic  polynomial  of  $a$ is  $f_S(t)=\prod
(t-a_i)$  and the  characteristic polynomial  of $b$  is $g_S(t)=\prod
(t-b_j)$.  (2) $a^{-1}b$ is identity  on a co-dimension one subspace of
$K^n$. \\

Then $\rho '$ is equivalent to $\rho$.
\end{theorem}

\begin{proof}  Put $D'=a-b$  and Let  $W$ be  the kernel  of  $D'$. By
assumption, $W$ has co-dimension one in $V$. Write
\[X=\cap _{i=0}^{n-2}a^{-i}W.\] Since $X$  is an intersection of $n-1$
hyperplanes in $V$, $X$ is non-zero.  Let $v\in X$, with $v\neq 0$. \\

We claim that $v$ is cyclic for the action of $a$.  Suppose not. Then,
$v,av, \cdots,  a^{n-1}v$ are linearly  dependent. We then  claim that
$a^{n-1}v$ is a linear combination of $v,av, \cdots, a^{n-2}v$: \\

Suppose $v,av,  \cdots, a^{n-2}v$  are already linearly  dependent. By
applying a suitable  power of $a$ to a  linear dependence relation, we
see  that $a^{n-1}v$  is a  linear combination  of the  vectors $v,av,
\cdots, a^{n-2}v$. \\

Suppose $v,av,  \cdots, a^{n-2}v$ are linearly  independent. Since the
vectors $v,av,  \cdots, a^{n-1}v$  are linearly dependent,  it follows
that   $a^{n-1}v$  is   a   linear  combination   of  $v,av,   \cdots,
a^{n-2}v$. \\

Since  $f_S(a)=0$, it  follows that  the  span $E$  of $v,av,  \cdots,
a^{n-2}v$  is  $a$  stable.   Since  all the  vectors  $v,av,  \cdots,
a^{n-2}v$ lie in  $W$ by the definition of $X$,  it follows that $a=b$
on $E$  and hence $E$  is stable under  $\Gamma$. Since $E\neq  0$, it
follows that the  characteristic polynomial of $a=b$ on  $E$ are equal
and  have  a  common  eigenvalue, contradicting  the  assumption  that
$a_i\neq b_j$ for any $i,j$. Therefore, $v$ is a cyclic vector for the
action of $a$. \\

Hence $v,av,  \cdots, a^{n-1}v$  is a basis  for $V$. It  follows that
with respect to this basis, the  matrix of $A$ is the companion matrix
of $f_S$. \\

By  the construction  of $X$,  we have  $a^iv\in W$  for  $i\leq n-2$.
Therefore, $ba^iv=aa^iv=a^{i+1}v$ for $i\leq n-2$.  Induction on $i\leq
n-2$ shows that  $a^iv=b^iv$ for all $i\leq n-1$.  With respect to the
basis  $v,av,  \dots, a^{n-1}v$  of  $V$, the  matrix  of  $a$ is  the
companion matrix of $f_S(t)$, and  that of $b$ is the companion matrix
of $g_S(t)$.  This completes the proof of the Theorem.
\end{proof}

\newpage

\section{Results of Beukers-Heckman}

The  Zariski  closure  of  the  hypergeometric  also  has  a  pleasant
description. This is described  by Beukers and Heckman \cite{Beu-Hec}.
For ease  of exposition, we assume  that the roots of  $f(x),g(x)$ are
roots of  unity (i.e.  $\a_j, \b  _k$ are rational numbers),  and that
$f,g$  are products  of  cyclotomic  polynomials.  Then  $f(x),g(x)\in
\Z[x]$.  Moreover, $f(0)=\pm 1$ and $g(0)=\pm 1$. We recall from the previous section that the monodromy group $H(f,g)$ is generated by the companion matrices $A,B$ of the polynomials $f,g$ respectively. 

\subsection{The Finite Case} 

We may assume that  the numbers $\a _j$ and $\b _k$  lie in the closed
open interval  $[0,1)$.  We  say that the  numbers $\a_j$ and  $\b _k$
``interlace''  if between  any two  $\a  _j$ there  is a  $\b _k$  and
conversely. \cite{Beu-Hec} give a criterion for the monodromy group to
be finite in terms of the parameters $\a,\b$:

\begin{theorem}    (Beukers-Heckman)    The    hypergeometric    group
corresponding to the  parameters $\a _j, \b _k$ is  finite if and only
if the parameters $\a_j$ and $\b _k$ interlace.
\end{theorem}

\subsection{Imprimitivity}    We   say    that    $f(X),g(X)$   are
``imprimitive'' if  there exists an integer $k\geq  2$ and polynomials
$f_1,g_1$  such that  $f(x)=f_1(x^k)$, $g(x)=g_1(x^k)$.  Otherwise, we
say that $f,g$ are a ``primitive'' pair. \\

We assume  henceforth that  $f,g$ form a  primitive pair and  that $\a
_j,\b _k$  do not  satisfy the interlacing  condition. Let $G$  be the
Zariski      closure       of      the      hypergeometric.      Write
$c=\frac{f(0)}{g(0)}$. Then $c=\pm 1$.

\begin{theorem}  \label{zariski}  (Beukers-Heckman)  Suppose that  the
roots of $f,g$ do not interlace, $f,g \in \Z[x]$ are primitive. \\

If $c=-1$ then the Zariski closure of the hypergeometric is isomorphic
to $O(n)$, the orthogonal group on $n$ variables. \\

If  $c=1$, then  the Zariski  closure is  the symplectic  group $Sp_n$
(under our assumptions, $n$ will necessarily be even).
\end{theorem}

We do not prove this theorem, since that would take us too far afield. 
We refer to \cite{Beu-Hec} for a proof of a more general result from 
which Theorem \ref{zariski} follows. 

\section{Symplectic Case} 

In  this section,  we  will assume  that the  Zariski  closure of  the
hypergeometric is a symplectic group; i.e. assume that $f,g\in \Z[x]$,
$f,g$  form a  primitive  pair and  that  the roots  of  $f,g$ do  not
interlace.  Assume that $f(0)=g(0)=1$ so that in Theorem \ref{zariski}
$c=1$. By  the result of  Beukers and Heckman  (Theorem \ref{zariski})
the   hypergeometric  $H(f,g)$   is  a   Zariski  dense   subgroup  of
$Sp_{\Omega}(\Z)$ for a non-degenerate  symplectic form $\Omega$ on $\Q
^n$.  It  is then  an interesting  question to  ask when  $H(f,g)$ has
finite index (i.e.  when is $H(f,g)$ an  arithmetic symplectic group).
There  is  no  complete  characterisation   but  some  cases  are  now
known.

\subsection{Arithmetic Groups}

See \cite{Sin-Ven} for the following result.

\begin{theorem}  \label{leading}  suppose that  $f,g$  are  as in  the
beginning    of   this    subsection   and    that   the    difference
$f-g=c_0+\cdots+c_dX^d$ with leading  coefficient $c=c_d\neq 0$; assume
that  $\mid  c  \mid  \leq  2$. Then  $H(f,g)$  has  finite  index  in
$Sp_{\Omega  }(\Z)$; thus  the hypergeometric  group is  an arithmetic
group.

\end{theorem}

As a family of  examples, consider, for an even integer $n$,  
\[f(X)=\frac{X^{n+1}-1}{X-1}=X^n+x^{n-1}+\cdots+X+1, \] 
\[g(X)=(X+1)\frac{X^n-1}{X-1}=X^n+2X^{n-1}+2X^{n-2}+\cdots    +2X+1.\]
The    difference   $f-g=-(X^{n-1}+X^{n-2}+\cdots+X)$    has   leading
coefficient $c=-1$  and hence  the hypergeometric $H(f,g)$  has finite
index in $Sp_{\Omega}$, by Theorem \ref{leading}.

\subsection{Thin Groups}

We recall the following definition (see \cite{Sar} for details)

\begin{defn} Let $\Gamma  \subset SL_n(\Z)$ be a subgroup  and $G$ its
Zariski closure in  $SL_n$. Then $G$ is defined over  $\Q$ and $\Gamma
\subset G(\Z)=G\cap SL_n(\Z)$.  We say that $\Gamma$ is  {\bf thin} if
$\Gamma$  has  infinite  index  in $G(\Z)$.  Otherwise,  we  say  that
$\Gamma$ is arithmetic.
\end{defn}

Note that the notion of thinness and of arithmeticity depends on the 
embedding $\Gamma \subset SL_n(\Z)$. 

It  is widely  believed  that most  hypergeometric  groups in  Theorem
\ref{zariski} are  thin. Theorem \ref{leading} says  however, that not
all the hypergeometrics are thin. There  is no general criterion as to
when the hypergeometric  are thin, except when the  Zariski closure is
$O(n,1)$  (see \cite  {FMS}).  In  the next  subsection,  we will  see
examples  of   thin  hypergeometrics   with  Zariski   closure  $Sp_4$
(constructed by Brav and Thomas \cite{Br-Th}).

\subsection{Fourteen Families}

Of  special   interest  are   the  hypergeometrics   corresponding  to
$f=(X-1)^4$ (i.e. when the monodromy around infinity is {\it maximally
unipotent}. The  number of choices  for $g$ are limited:  $g\in \Z [X]$
must be a product of cyclotomic polynomials, and must have degree $4$;
moreover, $g(1)\neq 0$. With these  constraints there are exactly $14$
choices for $g$. It is known  that the hypergeometric $H(f,g)$ is also
the  monodromy  group  associated  to  a  family  of  {\it  Calabi-Yau
threefolds}   fibering   over    the   thrice   punctured   projective
line. Moreover, these threefolds turn up in mirror symmetry. \\

In \cite{AESZ}  (see also  \cite{CEYY}), the  question of  thinness or
arithmeticity of these groups  was first raised. Theorem \ref{zariski}
and  its  proof enables  us  to  prove  that  the monodromy  group  is
arithmetic in three of these cases; later Singh \cite{Sin} adapted the
method to prove arithmeticity in four  more cases.  On the other hand,
Brav  and   Thomas  \cite{Br-Th}  have   proved  that  $7$   of  these
hypergeometric groups are thin. In particular , they prove

\begin{theorem}   \label{dworkfamily}   (Brav  and   Thomas)   Suppose
$f(X)=(X-1)^4$ and  $g(X)=\frac{X^5-1}{X-1}$. Then  the hypergeometric
group   $H(f,g)\subset   Sp_{\Omega}(\Z)$   (is   Zariski   dense   in
$Sp_{\Omega}$  and)  has  infinite   index  in  $Sp_{\Omega}(\Z)$;  in
particular, it is a thin monodromy group.
\end{theorem}

To  sum  up,  out  of  these fourteen  families,  $7$  are  arithmetic
(\cite{Sin-Ven}, \cite{Sin})  and $7$ are thin  \cite{Br-Th}.  Theorem
\ref{dworkfamily}  is  the first  example  of  a ``higher-rank''  thin
monodromy group whose  Zariski closure is a simple group.

\subsection{Questions} As  was mentioned  before, there is  no general
criterion  to determine  when a  group is  thin or  not. Consider  the
hypergeometric $H(f,g)$ associated to
\[  f(X)=(X-1)^n, \quad  g(X)=\frac{X^{n+1}-1}{X-1},\] say,  with even
$n$. It is easy to deduce from \cite{Beu-Hec} that $H(f,g)$ is Zariski
dense in $Sp_n$.  When $n=4$, this  is the group considered in Theorem
\ref{dworkfamily} and is  thin. However, for $n\geq6$ and  even, it is
not known  if the group  $H(f,g)$ has  infinite index in  the integral
symplectic group. In particular, let us consider the subgroup $\Gamma$ of
$Sp_6(\Z)\simeq Sp_{\Omega}(\Z)$  generated by the  companion matrices
of $(X-1)^6, \quad  \frac{X^7-1}{X-1}$.  It  is not known  if $\Gamma  $ has
finite index or not.

\end{document}